\documentclass[11pt,a4paper]{amsart}

\usepackage{times} 
\usepackage{amsmath} 
\usepackage{amssymb}  
\usepackage{amsthm}
\usepackage{latexsym}
\usepackage{amsfonts}
\usepackage{xcolor}
\usepackage{graphicx}

\newtheorem{thm}{Theorem}
\newtheorem{cor}[thm]{Corollary}
\newtheorem{lem}[thm]{Lemma}
\newtheorem{prop}[thm]{Proposition}
\theoremstyle{definition}
\newtheorem{defn}[thm]{Definition}
\newtheorem{rem}[thm]{Remark}

\numberwithin{equation}{section}
\allowdisplaybreaks

\renewcommand{\emptyset}{\varnothing}

\newcommand{\braces}[1]{{\rm (}#1{\rm )}}

\newcommand{\AND}{\quad\text{and}\quad}

\newcommand{\R}{\ensuremath{\mathbb R}}    
\newcommand{\C}{\ensuremath{\mathbb C}}    
\newcommand{\N}{\ensuremath{\mathbb N}}    
\newcommand{\T}{\ensuremath{\mathbb T}}    

\newcommand{\<}{\langle}
\renewcommand{\>}{\rangle}

\newcommand{\calB}{\mathcal B}

\newcommand{\calH}{\mathcal H}

\newcommand{\calK}{\mathcal K}

\newcommand{\calU}{\mathcal U}

\newcommand{\calX}{\mathcal X}         
\newcommand{\calY}{\mathcal Y}         
\newcommand{\calZ}{\mathcal Z}         

\newcommand{\veps}{\varepsilon}

\newcommand{\bmat}[4]
{
	\begin{bmatrix}
		#1 & #2\\
		#3 & #4
	\end{bmatrix}
}
\newcommand{\bvec}[2]
{
	\begin{bmatrix}
		#1\\
		#2
	\end{bmatrix}
}

\newcommand{\sbvec}[2]{\left[\begin{smallmatrix}#1\\#2\end{smallmatrix}\right]}

\renewcommand{\Im}{\operatorname{Im}}
\renewcommand{\Re}{\operatorname{Re}}

\newcommand{\dom}{\operatorname{dom}}

\newcommand{\ran}{\operatorname{ran}}

\newcommand{\Lra}{\Longrightarrow}
\newcommand{\Sra}{\Rightarrow}

\newcommand{\Lla}{\Longleftarrow}

\newcommand{\wto}{\rightharpoonup}

\newcommand{\ol}{\overline}

\newcommand{\PP}{\mathbb{P}}
\newcommand{\ti}{\tilde}
\renewcommand{\AND}{\qquad\text{and}\qquad}
\newcommand{\clran}{\ol{\ran}\,}

\begin{document}
\title{Strict dissipativity for generalized linear-quadratic problems in infinite dimensions}
\author[L.\ Gr\"une, F.\ Philipp, M.\ Schaller]{Lars Gr\"une$^{1}$, Friedrich Philipp$^{2}$, and Manuel Schaller$^{2}$}
	\thanks{}
	\thanks{$^{1}$Department of Mathematics, Faculty of Mathematics, Physics and Computer Sciences, University of Bayreuth, Bayreuth, Germany
		{\tt\small lars.gruene@uni-bayreuth.de}}
	\thanks{$^{2}$Institute for Mathematics, Faculty of Mathematics and Natural Sciences, Technische Universit\"at Ilmenau, Ilmenau, Germany
		{\tt\small \{friedrich.philipp, manuel.schaller\}@tu-ilmenau.de}.}

	\thanks{{\bf Acknowledgments: }F.\ Philipp was funded by the Carl Zeiss Foundation within the project \textit{DeepTurb---Deep Learning in und von Turbulenz}. 
		M.\ Schaller and L.\ Gr\"une were funded by the DFG (project numbers 289034702 and 430154635 for MS and 244602989 for LG).
}

\begin{abstract}
We analyze strict dissipativity of generalized linear quadratic optimal control problems on Hilbert spaces. Here, the term ``generalized'' refers to cost functions containing both quadratic and linear terms. We characterize strict pre-dissipativity with a quadratic storage function via coercivity of a particular Lyapunov-like quadratic form. Further, we show that under an additional algebraic assumption, strict pre-dissipativity can be strengthened to strict dissipativity. Last, we relate the obtained characterizations of dissipativity with exponential detectability.

\end{abstract}

\maketitle

\section{Introduction}
The notions dissipativity and strict dissipativity of dynamical systems, introduced by Jan Willems in the seminal works \cite{Will72a,Will72b}, are central in analyzing the behavior of dynamical systems as they ensure an energy balance for trajectories in terms of a storage function and a supplied energy. Whereas dissipativity states that the stored energy can not increase by more than the supplied energy, strict dissipativity as its strengthened version includes an additional dissipation term.

While passivity, a particular case of dissipativity has been well-established for decades as a foundation for designing stabilizing controllers \cite{ByIW91,Scha00}, the role of dissipativity for analyzing economic model predictive control (MPC), where the cost functional is not positive definite in the state was recognized quite recently \cite{DiAR10,AnAR12,Grue13,GruS14}. There, a particular form of strict dissipativity with the supply rate given by the cost functional, renders the optimal value function a Lyapunov function that itself can be leveraged to prove stability of the MPC-closed loop. Besides direct applications in MPC stability analysis, dissipativity plays an important role for optimally operated steady states see \cite{Muel14,MuAA15} and \cite{MuGA15}, and to the so-called turnpike property at an optimal equilibrium, see \cite{GruM16}.

Despite these various applications of strict dissipativity in optimal control, the first characterization of dissipativity of finite-dimensional generalized linear-quadratic optimal control problems (OCPs), that is, OCPs with quadratic and linear terms in the cost and linear dynamics, was provided in \cite{GruG18,GruG21}. In these works, a close relationship between detectability, where the output corresponds to the quadratic cost in the state, and strict dissipativity were given. Recently, in \cite{GrueneMuffSchaller}, we presented preliminary steps towards a characterization of strict dissipativity in the context of infinite-dimensional systems, however under very strong spectral assumptions on the underlying semigroup. In this work, we first establish a result for general dynamics that characterizes strict pre-dissipativity by means of a Lyapunov-like form inequality and then link the latter to detectability properties.

This paper is organized as follows. After providing some notation and the problem statement in Section~\ref{sec:probfor}, we establish alternative representations of strict dissipativity by means of an integral inequality or strict dissipativity of a problem with rotated stage cost in Section~\ref{sec:SD}. In Section~\ref{sec:QSF}, we focus on quadratic storage functions and provide in Theorem~\ref{t:char} a characterization of dissipativity for a given steady state by means of a Lyapunov-like form inequality. Whereas this result for a given steady state involves an algebraic condition, we show in Theorem~\ref{t:some} that this condition can be omitted if one aims for dissipativity at some steady state. Last, in Section~\ref{sec:detec}, we prove in Theorem~\ref{t:det} that the obtained form inequality is implied by exponential detectability and that the converse implication holds under an additional algebraic assumption on the Lyapunov operator.

\section{Generalized linear-quadratic optimal control problems}
\label{sec:probfor}
In this part we define the class of optimal control problems (OCPs) we shall study in this paper. Before, however, we require some notation.

\smallskip
\textbf{Notation.} Let $(\calH,\<\,\cdot\,,\,\cdot\,\>)$ be a real or complex Hilbert space. A self-adjoint operator $T$ in $\calH$ will be called
\begin{itemize}
\item {\em strictly positive}, if there exists $c>0$ such that $\<Tx,x\>\ge c\|x\|^2$ for all $x\in\dom T$;
\item {\em positive}, if $\<Tx,x\> > 0$ for all $x\in\dom T\backslash\{0\}$;
\item {\em non-negative}, if $\<Tx,x\>\ge 0$ for all $x\in\dom T$.
\end{itemize}
The square root of a non-negative self-adjoint operator $T$ will be denoted by $T^{1/2}$. For the existence and properties of this square root we refer to \cite[Theorem V.3.35]{Kato2013}.

The space of all bounded linear operators from a Hilbert space $\calH_1$ into another one $\calH_2$ will be denoted by $L(\calH_1,\calH_2)$. As usual, we set $L(\calH) := L(\calH,\calH)$. An operator $K\in L(\calH)$ is called {\em coercive}, if $K^*K$ is strictly positive. If $M\subset\calH$ is a closed linear subspace, then $\PP_M$ denotes the orthogonal projection onto $M$. The closure of the range $\ran T$ of a linear operator $T$ will be denoted by $\clran T$.

We say that a function $V : \calH\to\R$ is differentiable at $x_0\in\calH$ if there exists a continuous $\R$-linear functional $V'(x_0) : \calH\to\R$ (the derivative of $V$ at $x_0$) such that
$$
\lim_{x\to x_0}\frac{V(x) - V(x_0) - V'(x_0)(x-x_0)}{\|x-x_0\|} = 0.
$$
As usual, $V$ is called differentiable if it is differentiable at each $x\in\calH$. Due to Riesz' representation theorem, the derivative has the form $V'(x)h = \Re\<h,g(x)\>$, $h\in\calH$, with a unique $g(x)\in\calH$. We will frequently identify $g(x)$ with $V'(x)$. If $V$ is differentiable, we say that $V$ is $C^1$ if the function $g : \calH\to\calH$ is continuous. If $V(x) = \<Px,x\> + 2\Re\<x,w\>$ with a self-adjoint operator $P\in L(\calH)$ and a vector $w\in\calH$, then $V$ is $C^1$ with $V'(x)h = 2\Re\<h,Px+w\>$.

{\bf The OCP.} Throughout this paper, we assume that $A$ is a generator of a $C_0$-semigroup on the Hilbert space $\calH$. The control space $\calU$ is also assumed to be a Hilbert space $(\calU,\langle\cdot,\cdot\rangle)$ with the same notation for the scalar product. We consider the infinite-dimensional linear quadratic problem
\begin{equation}\label{e:ocp}
\min_{u\in L^2(0,T;\calU)}\int_0^T\ell(x,u)\,dt\,\,\,\mathrm{s.t.}\, \dot x = Ax + Bu,\quad x(0) = x_0
\end{equation}
with cost function $\ell : \calH\times\calU\to\R$, defined by
$$
\ell(x,u) := \|Cx\|^2 + \|Ku\|^2 + 2\Re\<z,x\> + 2\Re\<v,u\>.
$$
Solutions of $\dot x = Ax + Bu$ will be considered in the mild sense. Here, $C\in L(\calH,\calY)$, $K\in L(\calU,\calZ)$, $B\in L(\calU,\calH)$, $z\in\calH$, and $v\in\calU$, where $\calY$ and $\calZ$ are Hilbert spaces. We assume that there is some $c_K>0$ such that
\begin{equation}\label{e:KB}
\|Ku\|\ge c_K\|Bu\|,\qquad u\in\calU.
\end{equation}
This condition is equivalent to $\ran B^*\subset\ran K^*$, see \cite[Theorem IV.2.2]{Zabczyk} and is trivially satisfied if $K$ is coercive.

A pair $(x_e,u_e)\in\dom A\times\calU$ will be called a {\em controlled equilibrium} (or {\em steady state}) of the control problem $\dot x = Ax + Bu$, if $Ax_e + Bu_e = 0$. It is easy to see that the OCP (\ref{e:ocp}) is equivalent to the following one:
\begin{equation}\label{e:ocp0}
\min_{u\in L^2(0,T;U)}\int_0^T\ti\ell(x,u)\,dt\,\,\,\mathrm{s.t.}\, \dot{x} = Ax + Bu,\,\, x(0) = x_0 - x_e,
\end{equation}
where
\begin{align*}
\tilde\ell(x,u) = \|Cx\|^2 + \|Ku\|^2 + 2\Re\<z + C^*Cx_e,x\> + 2\Re\<v + K^*Ku_e,u\>.
\end{align*}
By equivalence of (\ref{e:ocp}) and (\ref{e:ocp0}) it is meant that $(x,u)$ is optimal for (\ref{e:ocp}) if and only if $(x,u)-(x_e,u_e)$ is optimal for (\ref{e:ocp0}). The optimal values might differ from each other, however.

\section{Strict pre-dissipativity}
\label{sec:SD}
In this section, we shall introduce and discuss the notion of strict (pre-)dissipativity of the OCP (\ref{e:ocp}), which involves the notion of storage functions and the class $\calK$. The latter is defined by
\begin{align*}
\calK := \big\{\alpha : [0,\infty)\to[0,\infty)\,|\,\alpha\text{ is continuous}& \\
\text{and strictly increasing with }& \alpha(0)=0\big\}.
\end{align*}
A {\em storage function on $\calH$} is a $C^1$-function $V : \calH\to\R$.

The following definition of strict dissipativity in the realm of infinite-dimensional optimal control is directly adopted from the finite-dimensional setting, see, e.g., \cite{GrueneGuglielmi,GrueneMuffSchaller}.

\begin{defn}
We say that the linear-quadratic OCP (\ref{e:ocp}) is {\em strictly pre-dissipative} at a controlled equilibrium $(x_e,u_e)$, if there exist a storage function $V$ on $\calH$ and a dissipation rate $\alpha\in\calK$ such that for all $x\in\dom A$ and all $u\in\calU$ we have
\begin{equation}\label{e:strdiss}
V'(x)(Ax+Bu)\,\le\,\ell(x,u) - \ell(x_e,u_e) - \alpha(\|x-x_e\|).
\end{equation}
The OCP (\ref{e:ocp}) is called {\em strictly dissipative} at $(x_e,u_e)$ if it is strictly pre-dissipative with a storage function $V$ that is bounded from below.
\end{defn}

We close this section by providing two characterizations of strict dissipativity.

\begin{lem}\label{l:at0}
The OCP (\ref{e:ocp}) is strictly pre-dissipative \braces{dissipative} at $(x_e,u_e)$ with storage function $V$ if and only if the OCP (\ref{e:ocp0}) is strictly pre-dissipative \braces{dissipative} at $(0,0)$ with storage function $\ti V = V(\,\cdot\,+x_e)$.
\end{lem}
\begin{proof}
Let OCP (\ref{e:ocp}) be strictly pre-dissipative at $(x_e,u_e)$ with storage function $V$. Then for all $x\in\dom A$ and $u\in\calU$ we have
\begin{align*}
\ti V'(x)(Ax + Bu)
&= V'(x+x_e)\big(A(x+x_e) + B(u+u_e)\big)\\
&\le \ell(x+x_e,u+u_e) - \ell(x_e,u_e) - \alpha(\|x\|)\\
&= \|C(x+x_e)\|^2 + \|K(u+u_e)\|^2 + 2\Re\<z,x+x_e\> 
\\&\quad + 2\Re\<v,u+u_e\>
- \|Cx_e\|^2 - \|Ku_e\|^2 - 2\Re\<z,x_e\> 
\\&\quad - 2\Re\<v,u_e\> - \alpha(\|x\|)\\
&= \|Cx\|^2 + 2\Re\<Cx,Cx_e\> + \|Ku\|^2 + 2\Re\<Ku,Ku_e\> 
\\&\quad+ 2\Re\<z,x\> + 2\Re\<v,u\> - \alpha(\|x\|)\\
&= \|Cx\|^2 + \|Ku\|^2 + 2\Re\<z+C^*Cx_e,x\> 
\\&\quad + 2\Re\<v+K^*Ku_e,u\> - \alpha(\|x\|)\\
&= \ti\ell(x,u) - \alpha(\|x\|).
\end{align*}
Since $\ti\ell(0,0) = 0$, it follows that OCP \eqref{e:ocp0} is strictly pre-dissipative at $(0,0)$ with storage function $\ti V$. The proof of the opposite implication follows similar lines.
\end{proof}

In what follows, we denote the mild solution to
\begin{equation}\label{e:ivp}
\dot x = Ax + Bu,\quad x(0) = x_0\in\calH
\end{equation}
by $x_u(\,\cdot\,;x_0)$, i.e.,
$$
x_u(t;x_0) = T(t)x_0 + \int_0^t T(t-s)Bu(s)\,ds,
$$
where $T(\cdot)$ denotes the $C_0$-semigroup generated by $A$.

\begin{lem}
The OCP \eqref{e:ocp} is strictly pre-dissipative at $(x_e,u_e)$ with storage function $V$ and a continuous dissipation rate $\alpha\in\calK$ if and only if for any $x_0\in\calH$, any $u\in L^2_{\rm loc}([0,\infty),\calU)$, and any $0\le t_1 < t_2$ we have
\begin{align}\label{e:strdiss_sol}
\begin{split}
V(x_u(t_2;x_0)) &- V(x_u(t_1;x_0))\\
&\le\,\int_{t_1}^{t_2}\!\!\big[\ell(x_u(t;x_0),u(t)) - \ell(x_e,u_e) - \alpha(\|x_u(t;x_0) - x_e\|)\big]\,dt.
\end{split}
\end{align}
\end{lem}
\begin{proof}
$\Lra$: Assume that (\ref{e:ocp}) is strictly pre-dissipative at $(x_e,u_e)$ with storage function $V$ and $\alpha\in\calK$. Set $J := [t_1,t_2]$, $\tilde J := [0,t_2]$, and denote the $L^p$-norm on $\tilde J$ by $\|\cdot\|_p$, $p\in [1,\infty]$. Let $u\in L^2(J,U)$ and $x_0\in\calH$. Since $\dom A$ is dense in $\calH$, there exists a sequence $(z_n)\subset\dom A$ such that $z_n\to x_0$ as $n\to\infty$. Moreover, by the density of the step functions in $L^2(\tilde J,U)$ \cite[Proposition 23.2, p407]{Zeidler19902a}, see also \cite[Lemma 2.1]{Schiela2013}, $C^1(\tilde J,U)$ is dense in $L^2(\tilde J,U)$ so that we find a sequence $(u_n)\subset C^1(\tilde J,U)$ such that $u_n\to u$ in $L^2(\tilde J,U)$. We set $x(t) := x_u(t;x_0)$ and $x_n(t) := x_{u_n}(t;z_n)$, $t\ge 0$. For $t\in J$ we have
\begin{align*}
\|x_n(t) - x(t)\|
&= \left\|T(t)(z_n-x_0) + \int_0^t T(t-s)B(u(s)-u_n(s))\,ds\right\|\\
&\le M\left(\|z_n-x_0\| + \|B\|\|u_n-u\|_1\right),
\end{align*}
where $M = \sup\{\|T(t)\| : t\in [0,t_2]\}$. Hence, $x_n\to x$ uniformly on $J$. Moreover, $x_n$ is a classical solution of \eqref{e:ivp} by \cite[Theorem 3.1.3]{CurtainZwart}, i.e., $x_n$ is a continuously differentiable curve, i.e., $x_n \in C^1(0,T;X)\cap C(0,T;\dom (A))$ and $\dot x_n(t) = Ax_n(t) + Bu_n(t)$ holds for every $t$. Therefore, $V\circ x_n$ is continuously differentiable with
\begin{align*}
\tfrac d{dt}V(x_n(t))
&= V'(x_n(t))(Ax_n(t)+Bu_n(t))\,\\
&\le\,\ell(x_n(t),u_n(t)) - \ell(x_e,u_e) - \alpha(\|x_n(t)-x_e\|)
\end{align*}
by strict pre-dissipativity. Integrating this over $J$ gives
\begin{align*}
V(&x_n(t_2)) - V(x_n(t_1))\,\le\,\int_{t_1}^{t_2}\big[\ell(x_n,u_n) - \ell(x_e,u_e) - \alpha(\|x_n(t)-x_e\|)\big]\,dt.
\end{align*}
By continuity of $V$, $\alpha$ and $\ell$, $\|x_n-x\|_\infty\to 0$ and $\|u_n-u\|_2\to 0$ as $n\to\infty$, all terms in this inequality tend to the corresponding ones in \eqref{e:strdiss_sol}.

$\Lla$: Assume that the condition \eqref{e:strdiss_sol} holds and let $x_0\in\dom A$ and $u_0\in U$ be arbitrary. Then the mild solution $x := x_u(\,\cdot\,,x_0)$ corresponding to $u(t) := u_0$ for $t\ge 0$ is a classical solution satisfying \eqref{e:strdiss_sol} with $t_1=0$. Dividing \eqref{e:strdiss_sol} by $t_2$ and letting $t_2\to 0$ yields
\begin{align*}
V'(x_0)(Ax_0+Bu_0) &= \tfrac d{dt}V(x(t))\big|_{t=0}\,\le\,\ell(x_0,u_0) - \ell(x_e,u_e) - \alpha(\|x_0-x_e\|),
\end{align*}
which is strict pre-dissipativity of \eqref{e:ocp} at $(x_e,u_e)$.
\end{proof}

\section{Quadratic storage functions}
\label{sec:QSF}
In this paper, we are especially interested in quadratic storage functions for strict (pre-)dissipa\-ti\-vi\-ty. A function $V : \calH\to\R$ will be called {\em quadratic} if it is of the form
$$
V_{P,w}(x) := \<Px,x\> + 2\Re\<w,x\>
$$
with a self-adjoint operator $P\in L(\calH)$ and $w\in\calH$. With regard to strict dissipativity it is of particular interest to characterize the quadratic storage functions which are bounded from below.

\begin{lem}\label{l:sf_charac}
The quadratic function $V_{P,w}$ is bounded from below if and only if $P$ is non-negative and $w\in\ran P^{1/2}$. $V_{P,w}$ has a minimizer if and only if $w\in\ran P$, in which case $P^{-1}\{-w\}$ is the set of minimizers.
\end{lem}
\begin{proof}
Let $V = V_{P,w}$. If $P\ge 0$ and $w\in\ran P^{1/2}$, choose $v\in\calH$ such that $w = P^{1/2}v$. Then
\begin{align}\label{e:half}
V(x) = \|P^{1/2}x\|^2 + 2\Re\big\<P^{1/2}x,v\big\> = \|P^{1/2}x + v\|^2 - \|v\|^2 \ge -\|v\|^2.
\end{align}
Conversely, assume that $V$ is bounded from below. Then there exists $c\ge 0$ such that $V+c\ge 0$. It is clear that $P\ge 0$. Write $w = w_1+w_2$ with respect to the orthogonal decomposition
$$
\calH = \ker P\oplus\clran P.
$$
Then $w_1=0$. Indeed, otherwise $V(-\frac{c+1}{2\|w_1\|^{2}}w_1) + c = -1 < 0$. Let $\PP_n$ denote the spectral projection of $P$ corresponding to the interval $[1/n,\|P\|]$ (cf.\ Appendix \ref{as:sa}) and set $w_n := \PP_n w$. Then $w_n\in\ran P$ for any $n$ (see Lemma \ref{l:helpy}(a)) and $w\in\clran P$ implies that $w_n\to w$ as $n\to\infty$. Let $x_n := -P^{-1}w_n$. Then
$$
-c\le V(x_n) = \|P^{1/2}x_n\|^2 + 2\Re\big\<w_n,x_n\big\> = - \|P^{-1/2}w_n\|^2.
$$
Hence, $(P^{-1/2}w_n)$ is bounded and therefore possesses a weakly convergent subsequence, i.e., $P^{-1/2}w_{n_k}\wto z$ for some $z\in\calH$. But this implies $w_{n_k}\wto P^{1/2}z$ and therefore $w = P^{1/2}z\in\ran P^{1/2}$.

If $w\in\ran P$ and $v\in P^{-1}\{-w\}$, then $V(x) = \|P^{1/2}x - P^{1/2}v\|^2 - \|P^{1/2}v\|^2\ge -\|P^{1/2}v\|^2$ for all $x\in\calH$ and $x^* = v$ is a minimizer. Conversely, assume that $V$ has a minimizer $x^*$. Then $w = P^{1/2}v\in\ran P^{1/2}$ since $V$ is bounded below. WLOG, we may assume that $v\in\ol{\ran}P = \ol{\ran}P^{1/2}$, so that there exists a sequence $(x_n)\subset\calH$ such that $P^{1/2}x_n\to -v$. Hence, \eqref{e:half} shows that $\inf\{V(x) : x\in\calH\} = -\|v\|^2$. Therefore, $V(x^*) = -\|v\|^2$ and \eqref{e:half} implies $v = -P^{1/2}x^*$. Thus, $w = P^{1/2}v  = -Px^*\in\ran P$ and $x^*\in P^{-1}\{-w\}$.
\end{proof}

Note that the running cost
$$
\ell(x,u) = \left\<\bmat{C^*C}00{K^*K}\bvec xu,\bvec xu\right\> + 2\Re\left\<\bvec zv,\bvec xu\right\>
$$
is itself a quadratic storage function. Hence, the following corollary is readily derived from Lemma \ref{l:sf_charac}.

\begin{cor}\label{c:ellbb}
The running cost function $\ell$ is bounded from below if and only if $z\in\ran C^*$ and $v\in\ran K^*$.
\end{cor}
\begin{proof}
The claim follows from $\ran(C^*C)^{1/2} = \ran C^*$ and $\ran(K^*K)^{1/2} = \ran K^*$, see \cite[VI.2.7]{Kato2013}.
\end{proof}

\begin{lem}\label{l:quad_at0}
The OCP \eqref{e:ocp} is strictly pre-dissipative at $(x_e,u_e)$ with storage function $V_{P,w}$ if and only if the OCP \eqref{e:ocp0} is strictly pre-dissipative at $(0,0)$ with storage function $V_{P,\ti w}$, where $\ti w = w+Px_e$.
\end{lem}
\begin{proof}
For $x\in\calH$ we have
\begin{align*}
V_{P,w}(x+x_e)
&= \<P(x+x_e),x+x_e\> + 2\Re\<w,x+x_e\>\\
&= \<Px,x\> + 2\Re\<Px_e,x\> + \<Px_e,x_e\> + 2\Re\<w,x\> + 2\Re\<w,x_e\>\\
&= V_{P,\ti w}(x) + (\<Px_e,x_e\> + 2\Re\<w,x_e\>).
\end{align*}
Hence, the claim follows from Lemma \ref{l:at0}.
\end{proof}

The next proposition shows in particular that strict pre-dissipativity is subject to an algebraic constraint on the steady state and the linear factors in the stage cost.

\begin{prop}\label{p:vw}
For $P=P^*\in L(\calH)$ and $w\in\calH$ the following statements are equivalent:
\begin{enumerate}
\item[{\rm (i)}]  The OCP \eqref{e:ocp} is strictly pre-dissipative at $(x_e,u_e)$ with storage function $V_{P,w}$ and dissipation rate $\alpha\in\calK$.
\item[{\rm (ii)}] We have
\begin{align}\label{e:alg}
\begin{split}
&w+Px_e\in\dom A^*,\\
& A^*(w+Px_e) = z+C^*Cx_e,
\\& B^*(w+Px_e) = v+K^*Ku_e
\end{split}
\end{align}
and for all $x\in\dom A$ and all $u\in\calU$,
\begin{equation}\label{e:reduced}
2\Re\<Ax+Bu,Px\>\,\le\,\|Cx\|^2 + \|Ku\|^2 - \alpha(\|x\|).
\end{equation}
\end{enumerate}
\end{prop}
\begin{proof}
By Lemma \ref{l:quad_at0}, (i) is equivalent to
\begin{enumerate}
\item[(i')] The OCP \eqref{e:ocp0} is strictly pre-dissipative at $(0,0)$ with storage function $V_{P,\ti w}$, $\ti w = w+Px_e$.
\end{enumerate}
Moreover, $\ti\ell(0,0)=0$, and hence (i') means that for all $x\in\dom A$ and $u\in\calU$,
\begin{align}\label{e:hotz}
\begin{split}
2\Re\<Ax+Bu,Px+\ti w\>
&\le\,\|Cx\|^2 + \|Ku\|^2 + 2\Re\<z+C^*Cx_e,x\>\\
&\qquad\quad\quad\,\, +2\Re\<v+K^*Ku_e,u\> - \alpha(\|x\|).
\end{split}
\end{align}
(ii)$\Sra$(i'). The combination of \eqref{e:alg} and \eqref{e:reduced} gives \eqref{e:hotz}. Therefore, (i') follows.

\noindent (i')$\Sra$(ii). By setting $x=0$ in \eqref{e:hotz} we obtain
$$
\|Ku\|^2 + 2\Re\<u,v+K^*Ku_e - B^*\ti w\>\,\ge\,0
$$
for all $u\in\calU$. But this is only possible if $v + K^*Ku_e -B^*\ti w = 0$. Hence, \eqref{e:hotz} simplifies to
\begin{align*}
2\Re&\<Ax+Bu,Px\> + 2\Re\<Ax,\ti w\>\,
\\&\le\,\|Cx\|^2 + \|Ku\|^2 + 2\Re\<z+C^*Cx_e,x\> - \alpha(\|x\|),
\end{align*}
which holds for all $x\in\dom A$ and all $u\in\calU$. Thus, setting $u=0$ yields
$$
2\Re\<Ax,Px\> + 2\Re\<Ax,\ti w\>\,\le\,\|Cx\|^2 + 2\Re\<z+C^*Cx_e,x\>
$$
for all $x\in\dom A$. Let $a>0$ and replace $x$ by $ax$ in the last inequality, i.e.,
\begin{align*}
2a^2\Re\<Ax,Px\> &+ 2a\Re\<Ax,\ti w\>\,\le\,a^2\|Cx\|^2 + 2a\Re\<z+C^*Cx_e,x\>.
\end{align*}
Dividing this by $a$ and letting $a\to 0$ implies that $\Re\<Ax,\ti w\>\,\le\,\Re\<x,z+C^*Cx_e\>$ for all $x\in\dom A$. Replacing $x$ by $-x$ in this inequality shows that in fact $\Re\<Ax,\ti w\> = \Re\<x,z+C^*Cx_e\>$ holds for all $x\in\dom A$. And as
\begin{align*}
\Im\<x,z+C^*Cx_e\> = -\Re\<ix,z+C^*Cx_e\> = -\Re\<iAx,\ti w\> = \Im\<Ax,\ti w\>,
\end{align*}
we obtain $\<Ax,\ti w\> = \<x,z+C^*Cx_e\>$ for all $x\in\dom A$. But this means that $\ti w\in\dom A^*$ and $A^*\ti w = z+C^*Cx_e$. This shows that \eqref{e:alg} holds, and \eqref{e:reduced} follows by plugging \eqref{e:alg} into \eqref{e:hotz}.
\end{proof}

%

\begin{rem}
As the proof above shows, the condition on $z$ and $v$ is even necessary for ``ordinary'' (pre-)dissipativity (i.e.\ $\alpha = 0$) with supply rate $\ell$ (and output $y=x$).
\end{rem}

\begin{thm}\label{t:char}
For $P=P^*\in L(\calH)$ and a controlled equilibrium $(x_e,u_e)$ the following statements are equivalent:
\begin{enumerate}
\item[{\rm (i)}]  There exist $\gamma>0$ and $w\in\calH$ such that the OCP \eqref{e:ocp} is strictly pre-dissipative at $(x_e,u_e)$ with storage function $V_{\gamma P,w}$.
\item[{\rm (ii)}] We have
\begin{equation}\label{e:alg2}
\bvec{z+C^*Cx_e}{v+K^*Ku_e}\,\in\,\ran\bvec{A^*}{B^*}
\end{equation}
and the following Lyapunov-like form inequality:
\begin{align}
\begin{split}\label{e:ext_lyapunov}
\exists\eta,m>0\,&\forall x\in\dom A :\\
&\|Cx\|^2 - 2\eta\Re\<Ax,Px\>\,\ge\,m\|x\|^2.
\end{split}
\end{align}
\end{enumerate}
\end{thm}
\begin{proof}
(i)$\Sra$(ii). Condition \eqref{e:alg} in Proposition \ref{p:vw} implies \eqref{e:alg2} and setting $u=0$ in \eqref{e:reduced} yields
$$
\|Cx\|^2 - 2\gamma\Re\<Ax,Px\>\,\ge\,\alpha(\|x\|),\qquad x\in\dom A.
$$
Now, replace $x$ by $x/\|x\|$ in this inequality to obtain \eqref{e:ext_lyapunov} with $\eta = \gamma$ and $m=\alpha(1)$.

\noindent (ii)$\Sra$(i). It is no restriction to assume $\eta = 1$ in \eqref{e:ext_lyapunov}. For $\gamma\in (0,1)$, $x\in\dom A$, and $u\in\calU$ we have
\begin{align*}
\|Cx\|^2 + \|Ku\|^2 &- 2\Re\<Ax+Bu,\gamma Px\>\\
&= \|Cx\|^2 - 2\gamma\Re\<Ax,Px\> + \|Ku\|^2 - 2\gamma\Re\<PBu,x\>\\
&\ge (1-\gamma)\|Cx\|^2 + \gamma m\|x\|^2 + c_K^2\|Bu\|^2 - 2\gamma \|P\|\|Bu\|\|x\|\\
&\ge \left[c_k\|Bu\| - \frac{\gamma\|P\|}{c_K}\|x\|\right]^2 - \frac{\gamma^2\|P\|^2}{c_K^2}\|x\|^2 + \gamma m\|x\|^2\\
&\ge \gamma\left(m - \frac{\gamma\|P\|^2}{c_K^2}\right)\|x\|^2 = c\|x\|^2
\end{align*}
with some $c>0$ is $\gamma$ is sufficiently small. If we set $\alpha(t) := ct^2$, then 
this is \eqref{e:reduced} with $P$ replaced by $\gamma P$. Now, due to \eqref{e:alg2}, there exists $\ti w\in\dom A^*$ such that $A^*\ti w = z+C^*Cx_e$ and $B^*\ti w = v + K^*Ku_e$. Set $w := \ti w - \gamma Px_e$. Then also \eqref{e:alg} holds (with $P$ replaced by $\gamma P$), and therefore Proposition \ref{p:vw} implies that the OCP \eqref{e:ocp} is strictly pre-dissipative at $(x_e,u_e)$ with storage function $V_{\gamma P,w}$.
\end{proof}

Our next aim is to show that \eqref{e:ext_lyapunov} is essentially equivalent to strict pre-dissipativity at a suitably chosen steady state $(x_e,u_e)$. In order to find this steady state, we analyse the optimal steady state problem
\begin{align}\label{e:ssp}
\min_{x \in \dom A,\,u\in\calU} \ell(x,u)\qquad\text{s.t. }Ax+Bu = 0,
\end{align}
which is certainly equivalent to
\begin{align}\label{e:ssp2}
\min_{y\in \ker[A\;B]}\ell(y).
\end{align}

\begin{lem}\label{l:ssp}
Assume that $K$ is coercive and that \eqref{e:ext_lyapunov} holds. Then the optimal steady state problem \eqref{e:ssp} has a unique solution $y^* = (x_e,u_e)$. It satisfies
\begin{equation}\label{e:keinesonne}
\mathbb P_{\ker[A\;B]}\bvec{z+C^*Cx_e}{v+K^*Ku_e} = 0.
\end{equation}
\end{lem}
\begin{proof}
As $A$ is a closed operator, so is $[A\;B]$, which implies that $\calX := \ker[A\;B]$ is closed in $\calH\times\calU$. Also, for $y\in\calX$ we have $\ell(y) = \<Ty,y\> + 2\Re\<q,y\>$, where
$$
T = \mathbb P_\calX\bmat{C^*C}00{K^*K}\bigg|_{\calX}
\qquad\text{and}\qquad
q = \mathbb P_\calX\bvec zv.
$$
We shall now show that the non-negative operator $T\in L(\calX)$ is in fact strictly positive and thus boundedly invertible, so that the claim follows from Lemma \ref{l:sf_charac}. Indeed, \eqref{e:keinesonne} is equivalent to $y^* = -T^{-1}q$. If $T$ was not strictly positive, there would exist a sequence $(y_n)\subset\calX$ with $\|y_n\|=1$ for all $n\in\N$ and $\<Ty_n,y_n\>\to 0$ as $n\to\infty$. If we write $y_n = \sbvec{x_n}{u_n}$, this means that $\|x_n\|^2+\|u_n\|^2=1$, $Cx_n\to 0$, and $Ku_n\to 0$. Since $K$ is coercive, it follows that $u_n\to 0$. Further, $Ax_n + Bu_n = 0$ yields $Ax_n\to 0$, hence with \eqref{e:ext_lyapunov} we conclude $x_n\to 0$ as $n\to\infty$. But this contradicts $\|x_n\|^2 + \|u_n\|^2 = 1$.
\end{proof}

\begin{thm}\label{t:some}
Let $P = P^*\in L(\calH)$, assume that $K$ is coercive and that $\ran[A\;B]$ is closed. Then there exist $\gamma>0$ and $w\in\calH$ such that \eqref{e:ocp} is strictly pre-dissipative at some controlled equilibrium with storage function $V_{\gamma P,w}$ if and only if \eqref{e:ext_lyapunov} holds. In this case, the controlled equilibrium $(x_e,u_e)$ can be chosen as the minimizer of problem \eqref{e:ssp}.
\end{thm}
\begin{proof}
If \eqref{e:ocp} is strictly pre-dissipative at some controlled equilibrium with storage function $V_{\gamma P,w}$, then \eqref{e:ext_lyapunov} holds by Theorem \ref{t:char}. 
Conversely, assume that \eqref{e:ext_lyapunov} is satisfied. By Lemma \ref{l:ssp}, the unique solution $y^* = (x_e,u_e)\in\ker[A\;B]$ of the optimal steady state problem \eqref{e:ssp} satisfies \eqref{e:keinesonne} and hence
$$
\bvec{z+C^*Cx_e}{v+K^*Ku_e}\,\in\,\big(\!\ker[A\;B]\big)^\perp = \ran\bvec{A^*}{B^*}.
$$
Thus, \eqref{e:alg2} holds and Theorem \ref{t:char} yields the claim. This also shows that $(x_e,u_e)$ can always be chosen as the minimizer of problem \eqref{e:ssp}.
\end{proof}

\begin{rem}
If $A$ is bounded, the condition \eqref{e:ext_lyapunov} can be written as an operator inequality of Lyapunov type:
$$
PA + A^*P\,\le\,\eta^{-1}(C^*C - mI).
$$
However, if $A$ is unbounded, the operator $T = PA + A^*P$ is in general not self-adjoint. It is symmetric (i.e., $T\subset T^*$), but in general it is not densely defined. Its domain might even be trivial, see \cite{arlinskiitretter20}. 
\end{rem}

We close this section with some corollaries.

\begin{cor}\label{c:simplify}
The OCP \eqref{e:ocp} is strictly pre-dissipative at $(x_e,u_e)$ with a quadratic storage function if and only if \eqref{e:alg2} holds and there exist a self-adjoint operator $P\in L(\calH)$ and $m>0$ such that
$$
\|Cx\|^2 - 2\Re\<Ax,Px\>\ge m\|x\|^2
$$
holds for all $x\in\dom A$.
\end{cor}

\begin{cor}\label{c:simplify2}
The OCP \eqref{e:ocp} is strictly dissipative at $(x_e,u_e)$ with a quadratic storage function if and only if there exist a non-negative operator $P\in L(\calH)$ and $m>0$ such that
$$
[A\;\;B]^{-*}\bvec{z+C^*Cx_e}{v+K^*Ku_e}\,\cap\,\ran P^{1/2}\neq\emptyset
$$
and
$$
\|Cx\|^2 - 2\Re\<Ax,Px\>\ge m\|x\|^2,\qquad x\in\dom A.
$$
\end{cor}

%

\begin{cor}\label{c:simplify4}
If $\ran[A\;B]$ is closed, then the OCP \eqref{e:ocp} is strictly pre-dissipative at some controlled equilibrium with a quadratic storage function if and only if there exist $P=P^*\in L(\calH)$ and $m>0$ such that
$$
\|Cx\|^2 - 2\Re\<Ax,Px\>\ge m\|x\|^2,\qquad x\in\dom A.
$$
\end{cor}

\section{Strict dissipativity and exponential detectability}
\label{sec:detec}
In what follows, we shall call an operator $A$ {\em exponentially stable} if it generates an exponentially stable $C_0$-semigroup $T(\cdot)$, i.e., $T(\cdot)x\in L^2([0,\infty),\calH)$ for each $x\in\calH$. The following theorem contains a weaker characterization for exponential stability as in the standard literature. However, the proof remains essentially the same.

\begin{thm}\label{t:char_stab}
Let $A$ be a generator of a $C_0$-semigroup in $\calH$. Then the following statements are equivalent:
\begin{enumerate}
\item[{\rm (i)}]   $A$ is exponentially stable.
\item[{\rm (ii)}]  There exists a positive operator $P\in L(\calH)$ such that
$$
2\Re\<Ax,Px\> = -\|x\|^2,\qquad x\in\dom A.
$$
\item[{\rm (iii)}] There exist a non-negative operator $P\in L(\calH)$ and $c>0$ such that
$$
2\Re\<Ax,Px\>\,\le\,-c\|x\|^2,\qquad x\in\dom A.
$$
\end{enumerate}
\end{thm}
\begin{proof}
(i)$\Sra$(ii) follows from Theorem 5.1.3 in \cite{CurtainZwart} and (ii)$\Sra$(iii) is trivial. Assume that (iii) holds, let $x\in\dom A$ and define $w(t) := \<PT(t)x,T(t)x\>$, $t\ge 0$. Differentiating $w$ with respect to time gives $\dot w(t) = 2\Re\<PAT(t)x,T(t)x\>\le -c\|T(t)x\|^2$. Integrating this over $[0,t]$, we obtain $0\le w(t)\le w(0) - c\int_0^t\|T(s)x\|^2\,ds$ and hence $\int_0^t\|T(s)x\|^2\,ds\le \frac 1cw(0) = \frac 1c\<Px,x\>$ for $x\in\dom A$ and $t\ge 0$. As $\dom A$ is dense in $\calH$, this inequality can be extended to all $x\in\calH$, which shows that $A$ is exponentially stable.
\end{proof}

The pair $(A,C)$ is called {\em exponentially detectable} (cf.\ \cite[Definition 5.2.1]{CurtainZwart}), if there exists an operator $F\in L(\calY,\calH)$ such that $A + FC$ is exponentially stable.

In \cite{GrueneGuglielmi} it was proved that in finite dimensions strict pre-dissipativity (at some steady state) with a quadratic storage function is equivalent to detectability of the pair $(A,C)$. The next theorem partially extends this statement to the infinite-dimensional case.

\begin{thm}\label{t:det}
If $(A,C)$ is exponentially detectable, then there exists a non-negative operator $P\in L(\calH)$ such that \eqref{e:ext_lyapunov} holds. Conversely, if \eqref{e:ext_lyapunov} is satisfied with some $0\le P\in L(\calH)$ and $\ran C^*\subset\ran P$, then $(A,C)$ is exponentially detectable.
\end{thm}
\begin{proof}
Assume that $(A,C)$ is exponentially detectable. By definition, there exists $F\in L(\calY,\calH)$ such that $A + FC$ is exponentially stable. Hence, by Theorem \ref{t:char_stab} there exists a positive operator $P\in L(\calH)$ such that
$$
2\Re\<(A+FC)x,Px\> = -\|x\|^2,\qquad x\in\dom A.
$$
Let $\eta>0$. For $x\in\dom A$ we obtain
\begin{align*}
\|Cx\|^2 - 2\eta\Re\<Ax,Px\>
&= \|Cx\|^2 + \eta\big( 2\Re\<FCx,Px\> + \<x,x\> \big)\\
&= \left\< \big(C^*C + \eta C^*F^*P + \eta PFC + \eta I\big)x,x\right\>.
\end{align*}
Let $K := PF$. In the remainder of this proof it will be shown that the bounded self-adjoint operator
$$
T := C^*C + \eta C^*K^* + \eta KC + \eta I
$$
is strictly positive for a suitable $\eta>0$. Then \eqref{e:ext_lyapunov} follows. To this end, let $\veps\in(0,\frac 14\|K\|^{-2})$ and decompose the Hilbert space as $\calH = \calH_1\oplus\calH_2$ with $\calH_1 = \ran E([0,\veps])$ and $\calH_2 = \ran E((\veps,\infty))$, where $E$ denotes the spectral measure of the non-negative self-adjoint operator $C^*C$. Set $\PP_j := \PP_{\calH_j}$, $j=1,2$. With respect to this decomposition write
$$
T = \begin{bmatrix}T_{11} & T_{12}\\T_{21} & T_{22}\end{bmatrix}
$$
where $T_{ij} = \PP_i T|_{\calH_j}$, $i,j=1,2$. For $x\in\calH_1$ we have
\begin{align*}
\<T_{11}x,x\> = \|Cx\|^2 + \eta\<K^*x,Cx\> + \eta\<Cx,K^*x\> + \eta\|x\|^2\ge \eta\big(1-2\|K\|\sqrt{\veps}\big)\|x\|^2
\end{align*}
as $\|Cx\|\le\sqrt{\veps}\|x\|$ (see Lemma \ref{l:helpy}(b)). With our choice of $\veps$ we conclude that $\kappa := 1-2\|K\|\sqrt{\veps} > 0$ so that $T_{11}\ge \eta\kappa$.

Write $T = C^*C + \eta Z$ with $Z = C^*K^* + KC + I$. As both $\calH_1$ and $\calH_2$ are invariant under $C^*C$, it follows that
$$
T_{12} = \eta Z_{12},\quad T_{21} = \eta Z_{12}^*,\quad\text{and}\quad T_{22} = C^*C|_{\calH_2} + \eta Z_{22},
$$
where $Z_{ij} = \PP_iZ|_{\calH_j}$. The second Schur complement of the above decomposition of $T$ is the bounded operator in $\calH_2$ given by
$$
S = T_{22} - T_{21}T_{11}^{-1}T_{12} = C^*C|_{\calH_2} + \eta Z_{22} - \eta^2 Z_{12}^* T_{11}^{-1} Z_{12}.
$$
Now, $T_{11}\ge \eta\kappa$ implies $T_{11}^{-1}\le \frac 1{\eta\kappa}$ and hence $Z_{12}^* T_{11}^{-1} Z_{12}\le \frac{\|Z_{12}\|^2}{\eta\kappa}$. Therefore,
$$
S \ge C^*C|_{\calH_2} + \eta \left(Z_{22} - \tfrac{1}{\kappa}\|Z_{12}\|^2\right),
$$
which is strictly positive for $\eta>0$ sufficiently small as $C^*C|_{\calH_2}\ge\veps$. Therefore, the same holds for $T$ since
$$
\begin{bmatrix}T_{11} & T_{12}\\T_{21} & T_{22}\end{bmatrix} = 
\begin{bmatrix}I & 0\\T_{21}T_{11}^{-1} & I\end{bmatrix}
\begin{bmatrix}T_{11} & 0\\0 & S\end{bmatrix}
\begin{bmatrix}I & T_{11}^{-1}T_{12}\\0 & I\end{bmatrix}.
$$
This proves \eqref{e:ext_lyapunov}.

\smallskip
Assume now that \eqref{e:ext_lyapunov} holds with (WLOG) $\eta = 1$ and that $\ran C^*\subset\ran P$. Let $\calH_0 := \ol{\ran}P$ and $P_0 := P|_{\calH_0}\in L(\calH_0)$. Then $P_0$ is (possibly not boundedly) invertible such that $PP_0^{-1}y = y$ for all $y\in\ran P$. Next, the operator $F := -\frac 12 P_0^{-1}C^*$ is bounded by the closed graph theorem, i.e. $F\in L(\calY,\calH)$, and for $x\in\dom A$ we have
\begin{align*}
2\Re\<(A+FC)x,Px\> &= 2\Re\<Ax,Px\> - \<P_0^{-1}C^*Cx,Px\>\\
&= 2\Re\<Ax,Px\> - \|Cx\|^2\\
&\le\,-m\|x\|^2.
\end{align*}
Hence, $A+FC$ is exponentially stable by Theorem \ref{t:char_stab}.
\end{proof}

Combining Theorems \ref{t:some} and \ref{t:det} shows that exponential detectability implies strict pre-dissi\-pa\-ti\-vi\-ty. If $\ran C^*\subset\ran P$, then also the converse holds. We note that invertibility (or, equivalently, strict positivity) of $P$ is sufficient for this condition to hold.

\section{Conclusion}
We defined strict (pre-)dissipativity at controlled equilibria for linear-quadratic optimal control problems in infinite dimensions and characterized this notion in various ways. We investigated quadratic functions in terms of boundedness from below and the existence of minimizers and characterized strict pre-dissipativity with quadratic storage functions. It turned out that this property depends on an algebraic condition (see \eqref{e:alg2}) and an analytic one (see \eqref{e:ext_lyapunov}) which is a form version of a Lyapunov-type operator inequality. We also proved that this analytic condition is closely related to exponential detectability of the pair $(A,C)$. This connection will be the subject of future work. Moreover, we will investigate the differences between conditions for strict dissipativity and strict pre-dissipativity.

\bibliographystyle{abbrv}
\bibliography{references.bib}

\appendix
\section{Self-adjoint operators}\label{as:sa}
Let $A$ be a self-adjoint operator in a Hilbert space $\calH$. By the spectral theorem for self-adjoint operators in Hilbert spaces (see, e.g., \cite[Theorem VIII.6]{reedsimon1}), there exists an orthoprojection-valued measure $E$ on the Borel sigma algebra $\calB$ of $\R$ which has the following properties:
\begin{itemize}
\item $E(\R\backslash\sigma(A)) = 0$ and $E(\sigma(A)) = I$.
\item $E(\Delta_1\cap\Delta_2) = E(\Delta_1)E(\Delta_2) = E(\Delta_2)E(\Delta_1)$ for $\Delta_1,\Delta_2\in\calB$.
\item $E(\Delta)\calH$ is $A$-invariant for each $\Delta\in\calB$.
\item $\sigma(A|_{E(\Delta)\calH})\subset\sigma(A)\cap\Delta$ for closed $\Delta\in\calB$.
\item For each $x\in\dom A$,
$$
\<Ax,x\> = \int_\R t\,d\nu_x(t)\AND \|Ax\|^2 = \int_\R t^2\,d\nu_x(t),
$$
where $\nu_x$ is the measure $\nu_x(\Delta) := \|E(\Delta)x\|^2$, $\Delta\in\calB$.
\end{itemize}
In the finite-dimensional case (i.e., when $A\in\C^{n\times n}$ is a Hermitian matrix), the projection $E(\Delta)$ is the orthogonal projection onto the sum of eigenspaces corresponding to the eigenvalues of $A$ in $\Delta$. This is also the case if $A$ has only discrete spectrum in $\Delta$ (i.e., isolated eigenvalues of finite muliplicities). Here, we need the following simple implications from the above.

\begin{lem}\label{l:helpy}
The following statements hold:
\begin{enumerate}
\item[(a)] If $\Delta\in\calB$ is compact and $0\notin\Delta$, then $E(\Delta)\calH\subset\ran A$.
\item[(b)] For $x\in E([-r,r])\calH$, $r > 0$, we have $\<Ax,x\>\le r\|x\|^2$.
\end{enumerate}
\end{lem}
\begin{proof}
The claim (a) follows from the fact that $A_\Delta := A|_{E(\Delta)\calH}$ is boundedly invertible in $E(\Delta)\calH$. Hence $E(\Delta)\calH = \ran A_\Delta\subset\ran A$. For the proof of (b) set $\Delta = [-r,r]$ and observe that
\begin{align*}
\<Ax,x\> &= \int_\R t\,d\nu_x(t) = \int_\Delta t\,d\nu_x(t) \\&\le r\nu_x(\Delta) = r\|E(\Delta)x\|^2 = r\|x\|^2
\end{align*}
for $x\in E(\Delta)\calH$.
\end{proof}
\end{document}